\documentclass[11pt,a4paper]{amsart}
\usepackage{a4wide}

\usepackage{etoolbox}
\makeatletter
\def\@secnumfont{\mdseries}
\def\section{\@startsection{section}{1}%
 \z@{.7\linespacing\@plus\linespacing}{.5\linespacing}%
 {\normalfont\scshape\centering}}
\def\subsection{\@startsection{subsection}{2}%
 \z@{.5\linespacing\@plus.7\linespacing}{-.5em}%
 {\normalfont\bfseries}}

\patchcmd{\@thm}{\let\thm@indent\indent}{\let\thm@indent\noindent}{}{}
\patchcmd{\@thm}{\thm@headfont{\scshape}}{\thm@headfont{\bfseries}}{}{}

\makeatother

\DeclareFontFamily{U}{matha}{\hyphenchar\font45}
\DeclareFontShape{U}{matha}{m}{n}{
      <5> <6> <7> <8> <9> <10> gen * matha
      <10.95> matha10 <12> <14.4> <17.28> <20.74> <24.88> matha12
      }{}
\DeclareSymbolFont{matha}{U}{matha}{m}{n}
\DeclareFontSubstitution{U}{matha}{m}{n}
\DeclareMathSymbol{\acap}{2}{matha}{"58}
\DeclareMathSymbol{\acup}{2}{matha}{"59}

\usepackage[dvips]{graphicx}
\usepackage{hyperref,color,mathdots,accents,extarrows,bigints,enumitem,stmaryrd}
\usepackage[all]{xypic}
\usepackage[ansinew]{inputenc}
\usepackage{amsmath,amsfonts,amssymb}
\usepackage[all,graph]{xy}
\usepackage[shortcuts]{extdash}
\usepackage[numbers]{natbib}

\theoremstyle{plain}
\newtheorem{theorem}{Theorem}[section]
\newtheorem{lemma}[theorem]{Lemma}

\newtheorem{proposition}[theorem]{Proposition}
\newtheorem*{proposition*}{Proposition}

\newtheorem*{question*}{Question}

\theoremstyle{remark}
\newtheorem*{remark}{Remark}
\newtheorem*{definition}{Definition}

\newtheorem*{claim}{Claim}

\makeatletter
\def\@tocline#1#2#3#4#5#6#7{\relax
 \ifnum #1>\c@tocdepth 
 \else
 \par \addpenalty\@secpenalty\addvspace{#2}%
 \begingroup \hyphenpenalty\@M
 \@ifempty{#4}{%
 \@tempdima\csname r@tocindent\number#1\endcsname\relax
 }{%
 \@tempdima#4\relax
 }%
 \parindent\z@ \leftskip#3\relax \advance\leftskip\@tempdima\relax
 \rightskip\@pnumwidth plus4em \parfillskip-\@pnumwidth
 #5\leavevmode\hskip-\@tempdima
 \ifcase #1
 \or\or \hskip 1em \or \hskip 2em \else \hskip 3em \fi%
 #6\nobreak\relax
 \dotfill\hbox to\@pnumwidth{\@tocpagenum{#7}}\par
 \nobreak
 \endgroup
 \fi}
\makeatother

\def\Q{\mathbb{Q}} 
 
\def\Z{\mathbb{Z}} 
\def\R{\mathbb{R}} 

\def\C{\mathbb{C}}
\def\N{\mathbb{N}}

\newcommand{\smfrac}[2]{\mbox{\footnotesize$\displaystyle\frac{#1}{#2}$}}

\DeclareMathAlphabet{\mathbf}{OML}{cmm}{b}{it}

\def\bnm{\begin{enumerate}} 
\def\enm{\end{enumerate}}
\def\ba{\begin{array}} 
\def\ea{\end{array}} 
\def\bpp{\begin{pmatrix}}
\def\epp{\end{pmatrix}}

\def\ker{\operatorname{ker}}

\def\op{\operatorname}

\def\wti{\widetilde}

\def\id{\operatorname{id}}

\def\PD{\op{PD}}

\numberwithin{equation}{section}
\def\op{\operatorname}
\begin{document}

\title{Linking forms revisited}

\author{Anthony Conway}
\address{Universit\'e de Gen\`eve, Section de math\'ematiques, 2-4 rue du Li\`evre, Case postale 64 1211 Gen\`eve 4, Suisse}
\email{anthony.conway@unige.ch}

\author{Stefan Friedl}
\address{Fakult\"at f\"ur Mathematik\\ Universit\"at Regensburg\\   Germany}
\email{sfriedl@gmail.com}

\author{Gerrit Herrmann}
\address{Fakult\"at f\"ur Mathematik\\ Universit\"at Regensburg\\   Germany}
\email{gerrit.herrmann@mathematik.uni-regensburg.de}

\maketitle
\begin{abstract}
We show that the $\Q/\Z$-valued linking forms on rational homology spheres are (anti-) symmetric
and we compute the linking form of a 3-dimensional rational homology sphere in terms of a Heegaard splitting. Both results have been known to a larger or lesser degree, but it is difficult to find rigorous down-to-earth proofs in the literature.  
\end{abstract}

\section{Introduction} 
Let $M$ be an oriented $(2n+1)$-dimensional rational homology sphere, i.e.\ $M$ is an oriented topological manifold with $H_*(M;\Q)\cong H_*(S^{2n+1};\Q)$. 
 In Section~\ref{section:linking-form} we recall the definition of the \emph{linking form} 
\[ \lambda_M\colon H_n(M;\Z)\times H_n(M;\Z)\,\,\to\,\ \Q/\Z.\]
It follows easily from the definition that it is bilinear and non-singular.
This form, whose definition goes back to Seifert~\cite{Seifert, SeifertThrelfall}, has since then appeared frequently both in the study of high-dimensional manifolds~\cite{KervaireMilnor, WallKilling, WallClassification6} and in low dimensional topology~\cite{Livingston,Boyer, CassonGordon}.

The following proposition states a key property of linking forms.

\begin{proposition}\textbf{\emph{(Seifert 1935)}}\label{prop:linking-form-symmetric}
Let $M$ be a $(2n+1)$-di\-men\-si\-onal rational homology sphere.
If $n$ is odd, then the linking form $\lambda_M$ on $H_n(M;\Z)$ is symmetric,
otherwise it is anti-symmetric.
\end{proposition}

This proposition was first formulated in the 3-dimensional context by  Seifert~\cite[p.~814]{Seifert}. Since Seifert did not yet have the tools of singular homology and cohomology theory at his disposal, he could only give a somewhat informal proof.
 Another somewhat informal proof is implicitly given in \cite[p.~59-60]{GL}, where the linking form is calculated in terms of the intersection form on a bounding 4-manifold.
But to the best of our knowledge there are not many rigorous proofs for the proposition in the literature.

Linking forms have been generalized by Blanchfield and many others to more general coefficients, where the corresponding linking forms are also well-known to be hermitian. But there are  again very few rigorous proofs for these statements, in fact the only careful proof we are aware of is given in the recent paper by Powell~\cite{Powell}.

We  give a rigorous quick proof of Proposition~\ref{prop:linking-form-symmetric}. We only use cup and cap products and we expect that the same approach can be used to reprove the hermitianness statement of Powell~\cite{Powell}. To keep the paper short and readable we will not attempt to carry out this generalization. 

In the following, given coprime natural numbers $p$ and $q$ we denote by $L(p,q)$ the lens space $S^3/\sim$ where $\sim$ is the equivalence relation on $S^3$ that is generated by
\[ (z,w)\,\,\sim\,\, \big(ze^{2\pi i/p},we^{2\pi iq/p}\big).\]
We give $S^3$ the standard orientation and we endow $L(p,q)$ with the unique orientation that turns the projection map $S^3\to L(p,q)$ into an orientation-preserving map.
The following proposition recalls the arguably most frequently used calculation of linking forms on 3-manifolds.

\begin{proposition}\label{prop:linking-form-lens-space}
The linking form of the 3-dimensional lens space $L(p,q)$ is isometric to the form
\[ \ba{rcl} \Z_p\times \Z_p&\to & \Q/\Z\\
(a,b)&\mapsto & -\smfrac{q}{p}\cdot a\cdot b.\ea\]
\end{proposition}

This proposition is essential in the classification of lens spaces up to homotopy equivalence, in fact Whitehead~\cite{Whitehead41} showed that two lens spaces are homotopy equivalent if and only if their linking forms are isometric.

In the literature, except for the precise sign in the formula,  many proofs of 
Proposition~\ref{prop:linking-form-lens-space} or of equivalent statements can be found. In fact many textbooks in algebraic topology contain a proof, see e.g.\  \cite[p.~306]{Hatcher}, \cite[Chapter~69]{Munkres} and \cite[p.~364]{Bredon}, except that as far as we understand it, none of these proofs address the precise sign in the calculation. All these proofs work very explicitly with lens spaces and it is not evident how they generalize to other 3-manifolds.

We will now explain how to calculate the linking form of any rational homology sphere in terms of a Heegaard splitting. We will then see that this calculation gives in particular a proof of Proposition~\ref{prop:linking-form-lens-space}. 

Throughout this paper, given $g\in \N$ we adopt the following notation:
\bnm
\item We denote by $X_g$ a handlebody of genus $g$ and we equip it with an orientation.
We denote by $Z_g$ a copy of $X_g$.
\item We write $F_g=\partial X_g=\partial Z_g$. We equip $F_g$ with the orientation coming from the boundary orientation of $X_g$.
\item We denote by $a_1,\ldots,a_g,b_1,\ldots,b_g\in H_1(F_g;\Z)$ a symplectic basis for   $H_1(F_g;\Z)$ such that  $a_1,\ldots,a_g$ form a basis for $H_1(X_g;\Z)$. Recall that ``symplectic basis" means that the intersection form of $F_g$ with respect to this basis is given by the matrix
\[ \hspace{1cm} \bpp 0&I_g \\ -I_g&0\epp\]
where we denote by  $I_g$ the $g\times g$-identity matrix. (For purists a calculation of the intersection form of a surface using cup products and cap products can be found in \cite[Chapters~47 and~55]{Friedl}).
\item Given an orientation-reversing self-diffeomorphism $\varphi$ of the genus $g$ surface $F_g$ we write $M(\varphi):=X_g \cup_{\varphi} Z_g$ where we identify $x\in F_g=\partial X_g$ with $\varphi(x)\in \partial Z_g$. We give $M(\varphi)$ the orientation which turns both the inclusions $X_g\to M(\varphi)$ and $Z_g\to M(\varphi)$ into orientation-preserving embeddings.
Furthermore  we denote by
\[\hspace{1cm}
\bpp A_\varphi&B_\varphi\\C_\varphi&D_\varphi \epp\]
  the matrix  that represents $\varphi_*\colon H_1(F_g;\Z)\to H_1(F_g;\Z)$ with respect to the ordered basis $a_1,\dots,a_g,b_1,\dots,b_g$.
\enm

One of the first theorems in 3-manifold topology states that every  closed 3-manifold  can be written as $M(\varphi)$ for some $g$ and some orientation-reversing diffeomorphism $\varphi\colon F_g\to F_g$. (Here and throughout the paper all manifolds are understood to be compact, oriented and path-connected.)
The following theorem thus gives a calculation of the linking form for any 3-dimensional rational homology sphere.

\begin{theorem}
\label{thm:Main}
Let $g\in \N$ and let $\varphi\colon F_g\to F_g$ be an orientation-reversing diffeomorphism.  If~$M(\varphi)$ is a rational homology sphere, then $B_\varphi\in M(g\times g,\Z)$ is invertible and the linking form of $M(\varphi)$ is isometric to the  form
\begin{align*}
\Z^g / B_\varphi^T \Z^g \times \Z^g / B_\varphi^T \Z^g &\to\,\, \Q / \Z \\
(v,w) &\mapsto\,\, -v^T B_\varphi^{-1}A_\varphi w.
\end{align*}
\end{theorem}

\begin{remark}
\mbox{}
\bnm
\item  In Theorem~\ref{thm:Main2} we will state precisely what isomorphism $\Z /B^T_\varphi \Z^g\to H_1(M(\varphi);\Z)$ we use.
\item As we mentioned above, the previous calculations of linking forms that we are aware of do not address the sign question of the formula, i.e.\ they only determine the linking form up to a fixed sign. We tried exceedingly hard to determine the sign correctly. Nonetheless, one should take our sign with a grain of salt. After we first thought that we had definitely determined the correct sign, we (and our careful referee) found many more sign errors.
\item
One could make the case that the statement of Theorem~\ref{thm:Main} is at least implicit in \cite{Reidemeister34} as explained by Seifert \cite[p. 827]{Seifert}.
But the calculation provided in that paper is not very rigorous by today's standards and it is also very hard to decypher for a modern reader, even if the reader is able to understand arcane German. To the best of our knowledge we provide the first proof of Theorem~\ref{thm:Main} that is rigorous and that only uses singular homology and cohomology. Also, similar to  our proof of the symmetry of linking forms, we think that our approach to calculating linking forms can be generalized quite easily to compute twisted linking forms of a closed 3-manifold in terms of a Heegaard splitting.
\enm
\end{remark}

We now return to lens spaces. We denote by $X=Z=S^1\times D^2$ the solid torus and we write $F=\partial X=\partial Z$. We equip $S^1$, $S^1\times D^2$ and $F=\partial X=S^1\times S^1$  with the standard orientation. Note that with these conventions  $a=[S^1\times 1]$ and $b=[1\times S^1]$ form a symplectic basis, in the above sense, for the torus~$\partial X$.
Let $p,q\in \N$ be coprime. We pick $r,s\in \N$ such that $qr-ps=-1$.
We write
\[ A\,\,=\,\, \bpp q& \hfill p \\  \hfill  s&  r \epp\]
and we denote by $\varphi\colon F\to F$ the orientation-reversing diffeomorphism such that $\varphi_*$, with respect to the basis given by $a=[S^1\times 1]$ and $b=[1\times S^1]$, is represented by the matrix $A$. (Here $S^1\times 1$ and $1\times S^1$ are viewed as submanifolds with the obvious orientation coming from $S^1$.)
In \cite[Chapter~56]{Friedl} it is proved, in full detail, that there exists an \emph{orientation-preserving} diffeomorphism 
from $L(p,q)$ to $X\cup_\varphi Y$.  Theorem~\ref{thm:Main}  thus
says that the linking form of $L(p,q)$ is isometric to the form
\[ \ba{rcl} \Z/p\times \Z/p&\mapsto & \Q/\Z\\
(v,w)&\mapsto & -v \cdot \smfrac{q}{p}\cdot w.\ea\]

\begin{remark}
One of the ideas of the proof is to reduce the calculation of Poincar\'e duality of a 3-manifold to the well-known calculation of Poincar\'e duality of the Heegaard surface $F$ of $M(\varphi)=X\cup_F Z$. A similar approach has been used in \cite{Friedl-Powell} to reduce the calculation of the Blanchfield form of a knot to the Poincar\'e duality of a Seifert surface.
\end{remark}

\begin{remark}
Given an $(2n+1)$-dimensional manifold $M$ one can also define a linking form on the torsion submodule of $H_n(M;\Z)$. The same argument 
as in the proof of Proposition~\ref{prop:linking-form-symmetric} shows that it is symmetric. In the 3-dimensional context it should not be very hard to generalize 
Theorem~\ref{thm:Main} to the case of 3-manifolds that are not rational homology spheres.
\end{remark}

We could like to conclude this introduction with the following quote which we found in ~\cite[p.~21]{Kirby}: ``Think with intersections, prove with cup products.''
In low-dimensional topology, many papers dealing with intersection pairing shy away from working with cup and cap products, instead one often uses intuitive but arguably not entirely rigorous arguments.
Consequently, one goal of this paper 
is to convince readers that cup and cap products are
amazing objects: once one has gotten used to them, not only do they provide wonderful (and arguably the only) tools for proving certain statements, they can also be used to give efficient calculations. Finally we would like to point out that arguments using cup and cap products easily generalize to twisted coefficients which can no longer dealt with by using ``naive'' arguments.

\subsection*{Conventions.}
By a manifold we mean what is often called a topological manifold, i.e.\ we do not require the existence of a smooth structure. Furthermore all manifolds are understood to be compact, oriented and path-connected.

\subsection*{Organization.}
This paper is organized as follows. In Section~\ref{section:cup-cap} we recall basic facts on the  cup product and the cap product with coefficients. 
In Section~\ref{section:linking-form} we recall the definition of the linking form and 
in Section~\ref{sec:Symmetry} we provide the proof that linking forms are (anti-) symmetric. 
Finally in Section~\ref{sec:Proof} we provide the proof of Theorem~\ref{thm:Main}.

\subsection*{Acknowledgments.} 
The first author was supported by the NCCR SwissMap funded by the Swiss FNS. He also wishes to thank the university of Regensburg for its hospitality.
The second and the third author gratefully acknowledge the
support provided by the SFB 1085 `Higher Invariants' at the University of Regensburg,
funded by the Deutsche Forschungsgemeinschaft DFG.
We are very grateful to the referee for very quickly providing two very thorough reports with many thoughtful and helpful comments. We also wish to thank Mark Powell for providing useful feedback. Finally special thanks to Jae Choon Cha.

\section{Preliminaries}
\label{sec:Prelim}
This section recalls the definition of the linking form as well as some standard facts of algebraic topology. References include~\cite{Bredon, Munkres, Hatcher, WallKilling, KervaireMilnor, SeifertThrelfall}. 
\medbreak

Before we start out the discussion of the properties of the cup product and the cap product we make a few remarks on sign conventions:
\bnm
\item Bredon \cite{Bredon} defines the coboundary map as
$\delta_n=(-1)^{n+1} \partial_{n+1}^*$ whereas most other books, e.g.\ Munkres~\cite{Munkres} and Hatcher~\cite{Hatcher} define the coboundary map as $\delta_n=\partial_{n+1}^*$. We choose to follow the latter convention. 
These sign conventions influence some of the formulas, e.g.\ 
the diagram in Lemma~\ref{lem:munkres} commutes only up to the sign $(-1)^{k+1}$, whereas following the approach of Bredon the diagram  in Lemma~\ref{lem:munkres} would commute.
\item For the definition of the cup and cap product we follow the definitions used in Hatcher~\cite{Hatcher}.
\bnm
\item  Comparing \cite[p.~206]{Hatcher} and \cite[p.~328]{Bredon} one sees that for the cup product of cohomology classes in degrees $k$ and $l$ the definitions differ by the sign $(-1)^{kl}$.
\item  Comparing \cite[p.~239]{Hatcher} and \cite[p.~335]{Bredon} one sees that for the cap product of a cohomology  class of degree $k$ with a homology class in degree $l$ the definitions differ by the sign $(-1)^{k(l-k)}$.
\enm 
We refer to \cite[Chapter~51 and~53]{Friedl} we refer to a detailed discussion of the different sign conventions of cup products and cap products in the books by Bredon \cite{Bredon}, Dold \cite{Dold}, Hatcher \cite{Hatcher} and Spanier \cite{Spanier}. 
\enm
The above sign conventions are also the ones used in \cite{Friedl}. 
\subsection{The cup product and the cap product}\label{section:cup-cap}
Let $X$ be a topological space and let $G,H$ be abelian groups. 
In the following given $i=0,1,\dots$ we denote by $v_i=(0,\dots,1,\dots,0)$ the $i$-th vertex of the standard simplex and given points $w_0,\dots,w_{s}$ in $\R^m$ we denote by $[w_r,\dots,w_{r+s}]\colon \Delta^s\to \R^m$ the unique affine linear map that sends $v_i$ to $w_{i}$ for $i=0,\dots,s$.  The usual definition of the cup product as provided in \cite[p.~206]{Hatcher} generalizes to a cup product
\[ \ba{rcl}\acup_{\otimes} \colon C^k(X;G)\times C^l(X;H)&\to& C^{k+l}(X;G\otimes H)\\
(\varphi,\psi)&\mapsto & \left( \ba{rcl} C_{k+l}(X)&\to & G\otimes H\\
\sigma&\mapsto & \varphi(\sigma\circ [v_0,\dots,v_k])\otimes 
\psi(\sigma\circ [v_k,\dots,v_{k+l}])\ea\right).\ea
\]
A slightly lengthy but uneventful calculation shows, see e.g.\  \cite[Lemma~51.1]{Friedl}, that for $f\in C^k(X;G)$ and $g\in C^l(X;H)$ we have
\begin{equation}
\label{eq:CuoOtimesCoboundary}
\delta(f \acup_{\otimes} g)
\,\,=\,\,\delta( f)\acup_{\otimes} g+(-1)^{k} \cdot f\acup_{\otimes} \delta(g)\,\in\, C^{k+l}(X;G\otimes H).
\end{equation}
This implies that the above cup product on cochains  descends to a cup product
\[ \acup_{\otimes} \colon H^k(X;G)\times H^l(X;H)\,\,\to\,\, H^{k+l}(X;G\otimes H).\]
We denote by $\Theta \colon G\otimes H\to H\otimes G$ the obvious isomorphism. 
Then for $\varphi \in H^k(X;G)$ and $\psi\in H^l(X;H)$ the usual proof of the (anti-) symmetry of the cup product, see e.g.\ \cite[Proposition~51.7]{Friedl}, can be used to show that 
\begin{equation}
\label{eq:cup-symmetry}
\Theta_*(\varphi \acup_{\otimes} \psi)\,\,=\,\,(-1)^{kl} \cdot \psi\acup_{\otimes}\varphi\,\, \,\in\,\,H^{k+l}(X;H\otimes G).
\end{equation}

If $H=\Z$, then using the obvious isomorphism $\nu\colon G\otimes \Z\to G$ we obtain the cup product
\[ \acup \colon H^k(X;G)\times H^l(X;\Z)\,\,\to\,\, H^{k+l}(X;G\otimes \Z)\,\,\xrightarrow{\nu_*}\,\, H^{k+l}(X;G).\]
The same holds if $G=\Z$ and $H$ is some arbitrary abelian group.

Now let $G$ be an abelian group and let $(X,U)$ be a pair of topological spaces.
The definition of the cap product as provided in \cite[p.~239]{Hatcher} generalizes to a cap product
\[ \ba{rcl}\acap \colon C^k(X;G)\times C_l(X,U;\Z)&\to& C_{l-k}(X,U;G)\\
(\varphi,\sigma)&\mapsto & \varphi(\sigma\circ [v_0,\dots,v_k])\otimes 
\sigma\circ [v_k,\dots,v_{k+l}]
\ea
\]
which descends to a cap product 
\[ \acap \colon H^k(X;G)\times H_l(X,U;\Z)\,\,\to\,\, H_{l-k}(X,U;G).\]
If $X$ is path-connected, then we make the identification $H_0(X;G)=G$ via the augmentation map.
In this case we  refer to
\[ \ba{rcl} \langle\,\,,\,\,\rangle \colon H^k(X;G) \times H_k(X;\Z)&\to & H_0(X;G)=G\\
(\varphi,\sigma)&\mapsto & \langle \varphi,\sigma\rangle:=\varphi\acap \sigma\ea\]
as the \emph{Kronecker pairing}.

\begin{lemma}
\label{lem:PropertiesCap}
Let $G$ be an abelian group and let  $ (X,U) $ be a  pair of topological spaces.
\begin{enumerate}[font=\normalfont]
\item  Let  $f \colon (X,U) \to (Z,V)$ be a map of pairs.
If $\xi\in H^k(Z;G)$ and  $\sigma \in H_l(X,U;\Z)$, then 
 \[ \hspace{1cm} f_*(f^*(\xi) \acap \sigma)\,\, =\,\,\xi \acap f_*(\sigma)\,\,\in \,\,H_{l-k}(Z,V;G).\]
$($Hereby note that the map $f_*$ is the map on relative homology whereas $f^*$ denotes the map $f^*\colon H^k(Z;G)\to H^k(X;G)$ on absolute cohomology.$)$
\item  
If   $\varphi\in H^k(X;G)$, $\psi \in H^l(X;\Z)$ and  $\sigma\in  H_m(X,U;\Z)$, then
\[ \hspace{1cm} (\varphi \acup \psi) \acap \sigma\,\,=\,\,(-1)^{kl}\cdot \varphi \acap (\psi \acap \sigma)\,\,\in\,\, H_{m-k-l}(X,U;G).\]
\end{enumerate}
\end{lemma}

\begin{proof}
The first statement follows easily from the definitions. For the second statement, it follows immediately from the definitions that 
$(\psi \acup \varphi) \acap \sigma= \varphi \acap (\psi \acap \sigma)$, see e.g.\ \cite[Lemma~53.5]{Friedl} for details. The desired statement is now a consequence of the aforementioned (anti-) commutativity of the cup product.
\end{proof}

Now let $M$ be an $n$-dimensional manifold. (Recall that all manifolds are assumed to be compact, oriented and path-connected.) As usual we denote by $[M]\in H_n(M,\partial M;\Z)$ the fundamental class. Let $G$ be an abelian group. The Poincar\'e duality theorem says that the map 
\[ \ba{rcl} \acap [M] \colon H^k(M;G)&\to& H_{n-k}(M,\partial M;G)\\
\varphi&\mapsto &\varphi\acap [M] \ea
\]
is an isomorphism. We denote by  $\op{PD}_M^G\colon H_{n-k}(M,\partial M;G)\to H^k(M;G)$ the inverse. 

Before we relate the Poincar\'e duality on a manifold to Poincar\'e duality on its boundary we need to discuss conventions. Given an $n$-dimensional oriented manifold $M$ we give the boundary $\partial M$ the orientation which is defined by the convention, that at a point $P\in \partial M$ a basis $v_1,\dots,v_{n-1}\in T_P(\partial M)$ is a positive basis if $w,v_1,\dots,v_{n-1}\in T_PM$ is a positive basis, where $w$ is an outward pointing vector of $T_PM$. With this convention the following lemma holds.
(We refer to \cite[Chapter~40]{Friedl} for a more detailed discussion on sign conventions.)

\begin{lemma}\label{lem:boundary-convention}
Let $M$ be an $n$-dimensional  oriented manifold. We denote by 
\[ \partial \colon H_n(M,\partial M;\Z)\,\,\to\,\, H_{n-1}(\partial M;\Z)\]
 the connecting homomorphism of the pair $(M,\partial M)$. Then
\[ \partial [M]\,\,=\,\, [\partial M]\,\in \, H_{n-1}(\partial M;\Z).\]
\end{lemma}

The following proposition follows from combining~\cite[Theorem VI.9.2]{Bredon} with Lemma~\ref{lem:boundary-convention}. (Alternatively see also \cite[Proposition~55.22]{Friedl}.) Note that in this instance the different sign convention of Bredon does not affect the outcome.

\begin{proposition}\label{prop:bredon}
Let $M$ be an $n$-dimensional oriented manifold and let $G$ be an abelian group.
We denote by $k\colon \partial M\to M$ the inclusion map.
Then for any $p\in \N_0$ the following diagram commutes up to the sign $(-1)^p$:
\[ \xymatrix@C1.95cm@R0.5cm{ H_{n-p}(M,\partial M;G)\ar[d]^(0.45){\partial}&
H^{p}(M;G)\ar[d]^(0.45){k^*}\ar[l]_-{\acap [M]}
 \\
 H_{n-p-1}(\partial M;G)&H^p(\partial M;G)\ar[l]_-{\acap[\partial M]}.}
\]
\end{proposition}

\subsection{The definition of linking form on rational homology spheres}
\label{section:linking-form}
Let $X$ be a topological space.  We denote by $\beta \colon H^k(X;\Q / \Z) \to H^{k+1}(X; \Z)$ the \emph{Bockstein homomorphism} which arises from the short exact sequence $0 \to \Z \to \Q \to \Q / \Z \to 0$ of coefficients. 
We define similarly the Bockstein homomorphism $\beta\colon H_k(X;\Q/\Z)\to H_{k-1}(X;\Z)$. 

\begin{lemma}\label{lem:munkres}
Let $Z$ be an $m$-dimensional compact manifold.
For any $k\in \{0,\dots,m-1\}$ the diagram
\[ \hspace{1cm} \xymatrix@C1.52cm@R0.5cm{ 
 H_{m-k-1}(Z;\Z)& H^{k+1}(Z;\Z) \ar[l]_-{\acap [Z]}\\
 H_{m-k}(Z;\Q/\Z) \ar[u]^(0.45){\beta}&H^k(Z;\Q/\Z)\ar[u]^(0.45){\beta} \ar[l]_-{\acap [Z]}\\
}\]
commutes up to the sign $(-1)^{k+1}$.
\end{lemma}

\begin{proof}
The lemma is basically  \cite[Lemma~69.2]{Munkres}, except that in the reference the sign is not specified. 
The sign comes from the following general fact:
Let $X$ be a topological space and let $G$ be an abelian group.
Furthermore let $\varphi\in C^k(X;G)$ and let $\sigma\colon \Delta^l\to X$ be a singular $l$-simplex. If $k\leq l$, then a straightforward calculation shows that
\[ \partial (\varphi\acap \sigma)\,\,=\,\, (-1)^{k+1} \cdot \delta \varphi\acap \sigma\,+\, (-1)^k\cdot (\varphi\acap \partial \sigma).\]
(If one takes the different sign conventions into account, this equality is exactly  \cite[Proposition~VI.5.1]{Bredon}.)
In our case $\sigma$ is a cycle that represents the fundamental class of $Z$. It is now clear that in our diagram the sign $(-1)^{k+1}$ appears. We leave the details of the precise argument to the reader.
\end{proof}

Now let $M$ be an $(2n+1)$-dimensional rational homology sphere with $n\geq 1$.  
In this case the Bockstein homomorphisms in homology and cohomology in dimension $n$ are in fact isomorphisms.
We denote by $\Omega$ the composition
\[
\ba{rcl} H_n(M;\Z)  \,\,\xrightarrow{\op{PD}_M^{\Z}}\,\, H^{n+1}(M;\Z) \,\,\xrightarrow{\beta^{-1}} \,\,H^n(M;\Q / \Z)&\xrightarrow{\op{ev}}& \op{Hom}_\Z(H_n(M;\Z),\Q / \Z)\\
\varphi&\mapsto & (\sigma \mapsto \langle \varphi,\sigma\rangle).\ea\]
of Poincar\'e duality, the inverse Bockstein and the Kronecker evaluation map.

\begin{definition}
\label{def:Linking}
The \emph{linking form} of a $(2n+1)$-dimensional rational homology sphere $M$ is the form
$$ \lambda_M \colon H_n(M;\Z) \times H_n(M;\Z) \to \Q / \Z$$
defined by $\lambda_M(a,b)=\Omega(a)(b).$
\end{definition}

We summarize some key properties of the linking form in the following lemma.

\begin{lemma}\label{lem:basics-linking-form}
Let $M$ be a  $(2n+1)$-dimensional rational homology sphere. Then the following statements hold:
\bnm[font=\normalfont]
\item $\lambda_M$ is bilinear and non-singular $($i.e.\ $\Omega$ is an isomorphism$)$,
\item  given $a$ and $b$ in $H_n(M;\Z)$, we have
$$ \lambda_M(a,b)
\,\,=\,\,\big\langle (\beta^{-1} \circ \op{PD}_M^\Z)(a) \acup \op{PD}_M^\Z(b),[M] \big\rangle, $$
\item if $n$ is odd, then the linking form $\lambda_M$ is symmetric, otherwise it is anti-symmetric.
\enm
\end{lemma}

\begin{proof}
It is clear that  $\lambda_M$ is bilinear. To show that $\lambda_M$ is non-singular we need to show that all three homomorphisms in the definition of $\Omega$ are isomorphisms. We only have to argue that the last homomorphism is an isomorphism, but this in turn is an immediate consequence of the  universal coefficient theorem and the fact that $\Q / \Z$ is an injective $\Z$-module.
 
We turn to the proof of (2). By the definition of the Kronecker pairing we have 
$$\big\langle (\beta^{-1} \circ \op{PD}_M^\Z)(a) \acup \op{PD}_M^\Z(b),[M] \big\rangle \,\,= \,\,\big((\beta^{-1} \circ \op{PD}_M^\Z)(a)  \acup \op{PD}_M^\Z(b)\big) \acap [M]. $$
Next, using  Lemma~\ref{lem:PropertiesCap} (2) and the fact that by definition we have $\op{PD}_M^\Z(b) \acap [M]=b$, we deduce that this expression reduces to  $(\beta^{-1} \circ \op{PD}_M^\Z)(a) \acap b$. (If we look at  Lemma~\ref{lem:PropertiesCap} (2) carefully we see that officially a term  $(-1)^{n(n+1)}$ appears, but fortunately this equals $+1$.) Looking back at Definition~\ref{def:Linking}, this is nothing but the linking form applied to $a$ and $b$, as claimed.

We postpone the proof of (3) to the next section.
\end{proof}

Lemma~\ref{lem:basics-linking-form}
might remind the reader of the intersection form of even-dimensional manifolds. In fact, since the proof of Theorem~\ref{thm:Main} will relate the linking form of $M(\varphi)$ to the intersection form of the Heegaard surface $F$, we briefly recall the definition of this latter form. Namely, given a closed oriented surface $F$, the \emph{intersection form} of $F$ with rational coefficients
$$ Q_F \colon H_1(F;\Q) \times H_1(F;\Q) \to \Q $$
is defined as
\[ Q_F(x,y)\,\,:=\,\, \big\langle \op{PD}_F^\Q(x)\acup \op{PD}_F^\Q(y),[F]\big\rangle\,\,=\,\, \big( \op{PD}_F^\Q(x)\acup \op{PD}_F^\Q(y)\big)\acap [F].\]
It follows immediately from Lemma~\ref{lem:PropertiesCap} (2) that for $x,y\in H_1(F;\Q)$ we have
\begin{equation}\label{eq:intersection-form-surface} 
\ba{rcl} 
Q_F(x,y)&=&\big( \op{PD}^\Q_F(x)\acup \op{PD}^\Q_F(y)\big)\acap [F]\\
&=& 
- \op{PD}^\Q_F(x)\acap \big( \op{PD}^\Q_F(y)\acap [F]\big)\,=\, 
 -\op{PD}_F^\Q(x)\acap y\,=\, -\big\langle\op{PD}_F^\Q(x),y\big\rangle.\ea
 \end{equation}

\subsection{Symmetry of the linking form}
\label{sec:Symmetry}
In this section, we shall give a short algebraic proof that the linking form is (anti-) symmetric. The idea is to use the definition of the linking form in terms of the cup product, see Lemma~\ref{lem:basics-linking-form} (2).

Throughout this section we denote by $\nu \colon \Q / \Z \otimes \Z \to \Q / \Z$
and  $\nu \colon \Z\otimes \Q / \Z  \to \Q / \Z$ the obvious isomorphisms.
Now  recall that by definition we can decompose the cup product $\acup$ as 
\begin{equation}
\label{eq:CupOtimes}
 H^k(M;\Z) \times H^l(M;\Q / \Z) \xrightarrow{\acup_{\otimes}} H^{k+l}(M;\Z\otimes \Q / \Z) \xrightarrow{\nu_*} H^{k+l}(M;\Q / \Z).
\end{equation}

\begin{lemma}\label{lem:bockstein-cup}
Let $X$ be a topological space. For any $x\in H^k(X;\Q / \Z)$ and $y\in H^l(X;\Q / \Z)$, we have
$$ \nu_* (\beta( x) \acup_{\otimes} y )\,\,=\,\,(-1)^{k+1} \cdot \nu_*(x\acup_{\otimes} \beta( y))\,\in\, H^{k+l+1}(X;\Q/\Z).$$
\end{lemma}

\begin{proof}
We denote by $\rho$ the canonical projection from $\Q$ to $\Q / \Z$.
Pick $f$ in $C^k(X;\Q)$ and $g$ in $C^l(X;\Q)$ so that $[\rho_*(f)]=x$ and $[\rho_*(g)]=y$.  
The usual mild diagram chase in the definition of the Bockstein homomorphism
shows that there exist unique cocycles $\beta(f)$ in $C^{k+1}(X;\Z)$ and $\beta(g)$ in $C^{l+1}(X;\Z)$ which satisfy $\iota_*(\beta(f))=\delta(f)$ and $\iota_*(\beta(g))=\delta(g)$; here $\delta$ denotes the coboundary map and $\iota$ denotes the inclusion map $\Z \to \Q$. Using~(\ref{eq:CuoOtimesCoboundary}) together with the definition of $\beta(f)$ and $\beta(g)$, we have the following equality in $H^{k+l+1}(X;\Q\otimes \Q)$:
\begin{equation}
\label{eq:CupSymmetric}
0=[\delta(f\acup_{\otimes} g)] 
= [\delta( f)\acup_{\otimes} g+(-1)^k \cdot f\acup_{\otimes} \delta( g)]
=[\iota_*(\beta(f))\acup_{\otimes} g+(-1)^k \cdot f\acup_{\otimes} \iota_*(\beta(g))].
\end{equation}
In order to relate the right hand side of~(\ref{eq:CupSymmetric}) to the expressions which appear in the statement of the lemma, we consider the following commutative diagram of group homomorphisms
\begin{equation}
\label{eq:CoeffDiagram}
\xymatrix@C1.2cm@R0.6cm{ 
\Z\otimes \Q\ar[r]^{\iota\otimes \id}\ar[d]_(0.45){\id\otimes \rho}& \Q\otimes \Q \ar[d]^(0.45)\mu &\ar[l] \Q\otimes \Z\ar[d]^(0.45){\rho \otimes \id}\ar[l]_{\id\otimes \iota}  \\ 
\Z\otimes \Q / \Z\ar[dr]^{\nu} &\Q\ar[d]^(0.45) \rho & \Q / \Z\otimes \Z\ar[dl]_{\nu}\\
&\Q / \Z,&}
\end{equation}
where $\nu$ and $\mu$ stand for the obvious multiplication maps. Using~(\ref{eq:CupSymmetric}) and the commutativity of~(\ref{eq:CoeffDiagram}), we get the following equality in $H^{k+l+1}(X;\Q / \Z)$:
\[ \ba{rlll}0&=\,\,(\rho\circ \mu)_*\big(\big[i_*(\beta(f))\acup_{\otimes} g\,\,+\,\,(-1)^k\cdot f\acup_{\otimes} i_*(\beta(g))\big]\big)\\[0.05cm]
&=\,\,\big[\nu_*\big(\beta(f)\acup_{\otimes} \rho_*(g)\big)\,\,+\,\,(-1)^k\cdot \nu_*\big(\rho_*(f)\acup_{\otimes} \beta(g)\big)\big]\\
&=\,\,\nu_*\big([\beta(f)]\acup_{\otimes} [\rho_*(g)]\big)\,\,+\,\,(-1)^k\cdot\nu_*\big([\rho_*(f)]\acup_{\otimes} [\beta(g)]\big)\\
&=\,\,\nu_*(\beta(x)\acup_{\otimes} y)\,\,+\,\,(-1)^k\cdot\nu_*(x\acup_{\otimes} \beta(y)).
\ea\]
Note that the third equality follows from the fact that  $\beta(f)\in C^{k+1}(X;\Z)$, $\rho_*(g)\in C^k(X;\Q/\Z)$, $\rho_*(f)\in C^l(X;\Q/\Z)$ and $\beta(g)\in C^{l+1}(X;\Z)$ are cocycles. 
The lemma now follows immediately.
\end{proof}

We can now finally provide the proof of Proposition~\ref{prop:linking-form-symmetric}.
For the reader's convenience we recall the statement.
\\

\noindent \textbf{Proposition~\ref{prop:linking-form-symmetric}.} \emph{Let $M$ be a $(2n+1)$-dimensional rational homology sphere. If $n$ is odd, then the linking form $\lambda_M$ is symmetric, otherwise it is anti-symmetric.}\\

\begin{proof}
Given $a$ and $b$ in $H_n(M;\Z)$, we set $x:=\beta^{-1}(\PD_M^\Z(a))$ and $y:=\beta^{-1}(\PD_M^\Z(b))$. Using Lemma~\ref{lem:basics-linking-form} (2), the factorization described in~(\ref{eq:CupOtimes}) and Lemma~\ref{lem:bockstein-cup}, we obtain
\begin{equation}
\label{eq:PresqueFiniiiii}
\ba{rcl} 
 \lambda_M(a,b)&=& \big\langle (\beta^{-1} \circ \op{PD}_M^\Z)(a) \acup \op{PD}_M^\Z(b),[M] \big\rangle\\
 &=& \langle \nu_*(x \acup_{\otimes} \beta(y) ),[M]\rangle\,\, =\,\,(-1)^{n+1} \cdot  \langle \nu_*(\beta(x) \acup_{\otimes} y),[M]\rangle.\ea
 \end{equation}
Since $n(n+1)$ is even it follows from (\ref{eq:cup-symmetry}) that 
$$\lambda_M(a,b)=(-1)^{n+1}  \cdot \langle \nu_*(y\acup_{\otimes} \beta(x)),[M]\rangle.$$
Proceeding as in~(\ref{eq:PresqueFiniiiii}), this is nothing but $\lambda_M(b,a)$, which concludes the proof of the proposition.
\end{proof}

\section{Proof of Theorem~\ref{thm:Main}}
\label{sec:Proof}

Our proof of Theorem~\ref{thm:Main} decomposes into two main steps. First, we provide a convenient presentation of $H_1(M;\Z)$, then we compute the linking form. We recall some of the notation from the introduction and we add a few more definitions which shall be used throughout this chapter.
 \bnm
\item We denote by $X_g$ a fixed handlebody of genus $g$ and we equip it with an orientation. We  denote by $Z_g$ a copy of $X_g$ which we also view as an oriented manifold.
\item We write $F_g=\partial X_g=\partial Z_g$.
\item We denote by $a_1,\ldots,a_g,b_1,\ldots,b_g\in H_1(F_g;\Z)$ a symplectic basis for   $H_1(F_g;\Z)$ such that  $a_1,\ldots,a_g$ form a basis for $H_1(X_g;\Z)$. In particular the  intersection numbers are given by $a_i \cdot b_j=\delta_{ij},b_i \cdot a_j=-\delta_{ij}, a_i \cdot a_j =0$ and $b_i \cdot b_j=0$ for $i=1,\ldots,g$.
Note that this implies that  $b_1,\ldots,b_g$ represent the zero element in $H_1(X_g;\Z)$. By a slight abuse of notation we also denote by $a_i\in H_1(X_g;\Z)$ the image of $a_i$ under the inclusion induced map $H_1(F_g;\Z)\to H_1(X_g;\Z)$.
\item Sometimes we will use the bases of (3)  to make the identifications $H_1(F_g;\Z)=\Z^{2g}$ and  $H_1(X_g;\Z)=\Z^g$. Furthermore, since $Z_g$ is a copy of $X_g$ we can use the same basis as for $H_1(X_g;\Z)$ to make the identification $H_1(Z_g;\Z)=\Z^g$.
\item Given an orientation-reversing self-diffeomorphism $\varphi$ of the genus $g$ surface $F_g$ we write $M(\varphi):=X_g \cup_{\varphi} Z_g$ where we identify $x\in F_g=\partial X_g$ with $\varphi(x)\in \partial Z_g$. 
Furthermore  we denote by
\[
\bpp A_\varphi&B_\varphi\\C_\varphi&D_\varphi \epp\]
  the matrix that represents $\varphi_*\colon H_1(F_g;\Z)\to H_1(F_g;\Z)$ with respect to the ordered basis $a_1,\dots,a_g,b_1,\dots,b_g$.
If $\varphi$ is understood, then we drop it from the notation.
\item The following diagram summarizes the various inclusion maps arising in the subsequent discussion:
\[ \hspace{1cm} \xymatrix@C0.8cm@R0.3cm{ & F_g \ar[dl]_j \ar[dd]^i\ar[dr]^k &\\ X_g\ar[dr]_l&&Z_g\ar[dl]^m \\ &M.}\]
We give $F_g\subset M$ the orientation given by $F_g=\partial X_g$ and viewing $X_g$ as a submanifold of $M$. Note that with all of our conventions we have $[F_g]=-[\partial Z_g]$ if we view $Z_g$ as a submanifold of $M$.
\item 
If $g$ is understood, then we  drop it from the notation.
\enm


\subsection{A presentation for $H_1(M;\Z)$}
\label{sub:Presentation}
We start with an elementary lemma.

\begin{lemma}
\label{lem:MatrixProperties}
Let $\varphi$ be a self-diffeomorphism of $F=F_g$. 
We have
\[ \begin{pmatrix} A & B \\ C & D  \end{pmatrix}^{-1}\,\,=\,\,
\begin{pmatrix}\hfill  D^T & -B^T \\ -C^T &\hfill  A^T  \end{pmatrix}.\]
In particular we have $AB^T=BA^T$.
\end{lemma}

\begin{proof}
Since $\varphi_*$ is a symplectic automorphism of $H_1(F;\Z)$, it follows that the matrix $R:=\left(\begin{smallmatrix}A&B\\C&D\end{smallmatrix}\right)$ preserves the symplectic matrix $J:=\left(\begin{smallmatrix}0&I_g\\-I_g&0\end{smallmatrix}\right)$. In other words we have $R^TJR=J$ which immediately implies that $R^{-1}=J^{-1}R^TJ$.
The first statement now follows from an elementary calculation.  

The second statement follows from multiplying the matrix $R=\left(\begin{smallmatrix}A&B\\C&D\end{smallmatrix}\right)$ with the inverse we just calculated and considering the top right corner which necessarily needs to be the zero matrix.
\end{proof}

Using this lemma we can provide a presentation matrix for $H_1(M(\varphi);\Z)$.

\begin{proposition}
\label{prop:Presentation}
Let $\varphi$ be a self-diffeomorphism of $F=F_g$. Then the following statements hold:
\bnm[font=\normalfont]
\item 
The abelian group $H_1(M;\Z)$ is generated by $i_*(a_1),\ldots,i_*(a_g)$ and with respect to this generating set,  $B^T$ is a presentation matrix. More precisely, the homomorphism $\Z^g\to H_1(M;\Z)$ given by $e_r\mapsto i_*(a_r)$ is an epimorphism and its kernel is given by $B^T\cdot \Z^g$. 
\item If $M=M(\varphi)$ is a $3$-dimensional rational homology sphere, then $\det(B)\ne 0$, i.e.\ $B$ is invertible over the rationals. 
\enm
\end{proposition}

\begin{proof}
We denote by $\iota\colon \partial Z\to Z$ the inclusion map. Since all the spaces involved are connected, the Mayer-Vietoris sequence of $M=X \cup_F Z$ yields the exact sequence
\begin{equation}
\label{eq:MayerVietoris}
 H_1(\partial Z;\Z) \,\,\xrightarrow{\Big(\hspace{-0.1cm}\ba{c} \hfill \varphi^{-1}_*\\ -\iota_*\ea\hspace{-0.1cm}\Big)} \,\,\ba{c} H_1(X;\Z)\\\oplus\\  H_1(Z;\Z)\ea \,\, \xrightarrow{ l_*\oplus m_*}\,\, H_1(M;\Z)\,\, \to \,\,0.
\end{equation}
Recalling our choice of bases, we observe that the inclusion induced map $\iota_*\colon H_1(\partial Z;\Z)\to H_1(Z;\Z)$ is represented by the matrix $(I_g \ 0)$.
Furthermore, by Lemma~\ref{lem:MatrixProperties}  the map $\varphi_*^{-1}$ is represented by $(D^T \ -B^T)$. The map $\iota_*\colon H_1(\partial Z;\Z)\to H_1(Z;\Z)$ is evidently an epimorphism and thus we see that the exact sequence displayed in~(\ref{eq:MayerVietoris}) reduces to 
$$ \ker(\iota_*) \,\,\xrightarrow{\varphi_*^{-1}}\,\, H_1(X;\Z) \,\,\to \,\,H_1(M;\Z)\,\, \to\,\, 0.$$
Evidently, $\ker(\iota_*)=0\oplus \Z^g$.  Since $\varphi_*^{-1}$ is represented by the matrix $(D^T \ -B^T)$, we deduce that the restriction of $\varphi_*$ to $\ker(\iota_*)$ is represented by $-B^T$, as desired. This concludes the proof of the first statement.

The second statement of the proposition  is an immediate consequence of the first statement.
\end{proof}

\subsection{The computation of the linking form}
\label{sub:Computation}
Recall that we denote by  $i \colon F \to M$  the inclusion. The proof of Theorem~\ref{thm:Main} is based on the following observation. If we manage to find a map $\theta \colon H_1(M;\Z) \to H_1(F;\Q / \Z)$ which makes the diagram
\begin{equation}
\label{eq:Wanted}
\xymatrix@C1.4cm@R0.55cm{H_1(M;\Z) \ar[r]^-{\op{PD}_M^\Z} \ar[d]^(0.45)\theta& H^2(M;\Z) \ar[r]^-{\beta^{-1}} & H^1(M;\Q/ \Z)  \ar[d]^(0.45){i^*}\\
H_1(F;\Q / \Z)\ar[rr]^-{\op{PD}_F^{\Q/ \Z}} && H^1(F;\Q  / \Z) &  }
\end{equation}
commute, then we can reduce the calculation of the Poincar\'e duality in the 3-manifold $M$ to the much-better understood Poincar\'e duality of the surface $F$ and it will be fairly easy to compute the linking form.

Indeed, assuming such a map $\theta$ exists, we claim that the computation of 
\[ \lambda_M \circ (i_* \times i_*)\colon H_1(F;\Z)\times H_1(F;\Z)\,\to\, \Q/\Z\]
boils down to the computation of $\theta \circ i_*$. 
More precisely, for $v,w\in H_1(F;\Z)$ we apply successively the definition of the linking form, the naturality of the evaluation map (which is a consequence of  Lemma~\ref{lem:PropertiesCap} (1)) and the commutativity of~(\ref{eq:Wanted}) to obtain that:
\begin{align}
\label{eq:Reduction}
 \lambda_M(i_*(v),i_*(w))&=\,\,\big\langle (\beta^{-1} \circ \op{PD}_M^\Z \circ i_*)(v), i_*(w) \big\rangle_M
\,\,=\,\,\big\langle (i^* \circ  \beta^{-1} \circ \op{PD}_M^\Z \circ i_*)(v),w \big\rangle_F \\
&=\,\,\big\langle (\op{PD}_F^{\Q / \Z} \circ  \theta \circ i_*) (v), w \big\rangle_F\,\,\in\,\,\Q/\Z. \nonumber
\end{align}
Summarizing, the proof of Theorem~\ref{thm:Main} now decomposes into two steps: firstly, we define the map $\theta \colon H_1(M;\Z) \to H_1(F;\Q / \Z)$ (and check that it makes~(\ref{eq:Wanted}) commute) and secondly, we compute~$\theta \circ i_*$. To carry out the first step, define $\theta$ as the composition
\begin{equation}
\label{eq:Theta}
H_1(M;\Z) \,\xrightarrow{\,\beta^{-1}\,} \, H_2(M;\Q/\Z)\, \xrightarrow{\,p_*\,}\,  H_2(M,X;\Q/\Z) \,\xleftarrow{\,\cong\,}\, H_2(Z,F;\Q/\Z) \,\xrightarrow{\partial} \, H_1(F;\Q/\Z)
\end{equation}
of the following maps: the inverse homological Bockstein homomorphism, the map induced by the obvious map $p\colon (M,\emptyset) \to (M,X)$, the inverse of the excision isomorphism
(which is applicable since $(M,X,Z)$ is an excisive triad, full details can be found in \cite[Chapter~43]{Friedl}) and the connecting homomorphism of the long exact sequence of the pair $(Z,F)$ with $\Q/\Z$-coefficients. 

\begin{lemma}
\label{lem:ThetaWorks}
The homomorphism $\theta$ defined in~(\ref{eq:Theta}) makes~(\ref{eq:Wanted}) commute. More precisely, we have
$$\op{PD}_F^{\Q / \Z} \circ \theta\,\,=\,\, i^* \circ \beta^{-1} \circ \op{PD}_M^\Z\colon H_1(M;\Z)\,\to\, H^1(F;\Q/\Z).$$
\end{lemma}

\begin{proof}
We consider the maps of pairs $p\colon  (M,\emptyset) \to (M,X)$ and  $q \colon (Z,F) \to (M,X)$.
Note that our orientation conventions from the beginning  of the section implies that we have $p_*([M])=q_*([Z])$ and that $[\partial Z]=-[\partial X]=-[F]$.

Recall that by definition, capping with the fundamental class is the inverse of the Poincar\'e duality isomorphism. Keeping this in mind, the lemma will be proved if we manage to show that the following diagram commutes:
\[ \xymatrix@C1.57cm@R0.55cm{ H_1(M;\Z)\ar@<-1.2pc>@/_4pc/[ddddd]_{\theta}
 &&\ar[ll]_{\acap [M]}^\cong  H^2(M;\Z) \\
\ar[u]^\cong_{\beta}^\cong H_2(M;\Q/\Z)\ar[d]^{p_*}&& \ar[ll]_{\acap [M]}^\cong \ar[u]^(0.45){\beta}
 H^1(M;\Q/\Z)  \ar[dd]_{m^*}\ar@<1pc>@/^3pc/[dddd]^{i^*} \\
H_2(M,X;\Q/\Z)\\
H_2(Z,F;\Q/\Z)\ar[u]_(0.45){q_*}^\cong \ar[d]^(0.45){\partial} &&\ar[ll]_{\acap [Z]}^\cong H^1(Z;\Q/\Z)\ar[d]_{k^*}\\
H_1(F;\Q/\Z)\ar[d]^(0.45)= && \ar[ll]_{\acap- [\partial Z]}^\cong H^1(F;\Q/\Z)\ar[d]_(0.45)=\\
H_1(F;\Q/\Z) && \ar[ll]_{\acap [F]}^\cong H^1(F;\Q/\Z).}\]
Indeed, starting from the upper right corner and traveling to the lower left corner, the leftmost path produces the map $\theta \circ (\op{PD}_M^{\Z})^{-1}$, while the rightmost path produces the map~$(\op{PD}_F^{\Q / \Z})^{-1} \circ i^* \circ \beta^{-1}$.

The top square commutes by Lemma~\ref{lem:munkres}, to be precise, it commutes since in our case we have $(-1)^2=1$.
The third square from the top commutes  by Proposition~\ref{prop:bredon}.
(Note that we had to sneak in a minus sign in front of the $[\partial Z]$ to cancel the minus sign we would otherwise pick up from Proposition~\ref{prop:bredon}.)
 The bottom square
commutes since we had observed in the beginning of the proof that $[\partial Z]=-[F]$. Finally  the second square (or first and only pentagon, depending on your point of view), commutes by applying the first statement of Lemma~\ref{lem:PropertiesCap}. More precisely, 
applying  the first statement of Lemma~\ref{lem:PropertiesCap}  to the two maps 
 $p\colon  (M,\emptyset) \to (M,X)$ and  $q \colon (Z,F) \to (M,X)$
and using that  $p_*([M])=q_*([Z])$ we obtain that
for every $\varphi$ in $H^1(M;\Q / \Z)$, we have the following equality in $H_2(M,X;\Q/\Z)$:
$$ p_*(\varphi \acap [M])\,=\,p_*(p^*(\varphi) \acap [M])\,=\,\varphi \acap p_*([M])\,=\,\varphi \acap q_*([Z])\,=\,q_*(q^*(\varphi) \acap [Z])\,=\,q_*(m^*(\varphi) \acap [Z]).$$
Here the first equality  can easily give rise to confusion. The point is that $p\colon (M,\emptyset)\to (M,X)$ is a map of pairs of topological spaces which is the identity on the first entry. In Lemma~\ref{lem:PropertiesCap} (1) we could have distinguished in our notation between the maps of pairs of topological spaces and the maps on the two individual spaces but we declined to do so to keep the notation short. The same applies to the last equality, since the map $q\colon (Z,F)\to (M,X)$ of pairs of  topological spaces, when restricted to the first entry  is precisely the map $m$.
\end{proof}

In the remainder of this paper we use the following notation:
\bnm
\item We denote by  $i$ the inclusion  map $F\to M$.
\item We denote by $\rho \colon \Q \to \Q / \Z$  the canonical projection.
\item We denote by $\Phi_{\Z}\colon \Z^g\to H_1(F;\Z)$ the map that is given by $\Phi_{\Z}(e_r)={a}_r$, similarly we define $\Phi_{\Q}\colon \Q^g\to H_1(F;\Q)$ and  $\Phi_{\Q/\Z}\colon (\Q/\Z)^g\to H_1(F;\Q/\Z)$. We will use on several occasions  that for $\Z^g\subset \Q^g$ the maps $\Phi_{\Z}$ and $\Phi_{\Q}$ agree.
\item If in (3) we replace the $a_r$ by $b_r$ we obtain  maps that we denote by $\Psi_\Z$, $\Psi_{\Q}$ and $\Psi_{\Q/\Z}$.
\enm
The next proposition deals with the computation of $\theta \circ i_*\colon H_1(F;\Z)\to H_1(F;\Q/\Z)$ on the span of $a_1,\dots,a_g\in H_1(F;\Z)$. 

\begin{proposition}
\label{prop:ComputeTheta}
For any $v\in \Z^g$ the following equality holds:
\[(\theta\circ i_*)(\Phi_{\Z}(v))\,\,=\,\,-\Psi_{\Q/\Z}(B^{-1}Av)\,\,\in\,\, H_1(F;\Q/\Z).\]
\end{proposition}

\begin{proof}
In this proof we will mostly drop all inclusion maps from the notation, especially if we work on the chain level. We denote by $\wti{a}_1,\dots,\wti{a}_g$, $\wti{b}_1,\dots,\wti{b}_g$  singular chains in $F$ that represent $a_1,\dots,a_g,$ $b_1,\dots,b_g$.  Let $v=(v_1,\dots,v_g)\in \Z^g$. 
We denote by $\wti{\Phi}_{\Z}\colon \Z^g\to C_1(F;\Z)$ the map that is given by $\wti{\Phi}_\Z(e_r)=\wti{a}_r$ and we denote by  $\wti{\Psi}_{\Z}\colon \Z^g\to C_1(F;\Z)$ the map that is given by $\wti{\Psi}_\Z(e_r)=\wti{b}_r$. We make the obvious adjustments in the notation when we use other coefficients.

\begin{claim}\mbox{}
\bnm 
\item There exists $x\in C_2(X;\Q)$ with $\partial_\Q(x)=\wti{\Psi}_{\Q}(B^{-1}Av) \in C_1(F;\Q)\subset C_1(X;\Q)$. 
\item There exists $z\in C_2(Z;\Q)$ with $\partial_\Q(z)=\wti{\Phi}_\Z(v)-\wti{\Psi}_{\Q}(B^{-1}Av) \in C_1(F;\Q)\subset C_1(Z;\Q)$. 
\enm
\end{claim}

Note that  $\wti{\Psi}_\Q(B^{-1}Av)$ is a rational linear combination of $\wti{b}_1,\dots,\wti{b}_g$.
Since each $\wti{b}_r$ is null-homologous in $X$ we see that   $\wti{\Psi}_\Q(B^{-1}Av)$ is null-homologous in $C_*(X;\Q)$. This shows that there exists a singular 2-chain $x\in C_2(X;\Q)$ with $\partial_\Q (x)=\wti{\Psi}_\Q(B^{-1}Av)$. 

We make the usual identification $H_1(Z;\Z)=\Z^g$ and $H_1(Z;\Q)=\Q^g$ coming from the fact that $Z$ is a copy of $X$. Under this identification the map $k_*\circ \Phi_{\Q}\colon \Q^g\to H_1(Z;\Q)=\Q^g$ is by definition given by the matrix $A$ and the map $k_*\circ \Psi_{\Q}\colon \Q^g\to H_1(Z;\Q)=\Q^g$ is by definition given by the matrix $B$. Putting these two observations together we see that in $H_1(Z;\Q)=\Q^g$ we have the equality:
\[ k_*\big(\Phi_{\Q}(v)-\Psi_{\Q}(B^{-1}Av)\big)\,\,=\,\, (k_*\circ \Phi_\Q)(v)-
(k_*\circ \Psi_\Q)(B^{-1}Av)\,\,=\,\, Av-BB^{-1}Av\,\,=\,\,0.\]
Put differently, the singular 1-chain $\wti{\Phi}_{\Q}(v)-\wti{\Psi}_{\Q}(B^{-1}Av)=\wti{\Phi}_{\Z}(v)-\wti{\Psi}_{\Q}(B^{-1}Av)$ is null-homologous in $C_*(Z;\Q)$, i.e.\  there exists a singular 2-chain $z\in C_2(Z;\Q)$ with $\partial_\Q(z)=\wti{\Psi}_{\Z}(v)-\wti{\Psi}_\Q(B^{-1}Av)$.
This concludes the proof of the claim.

From the definition of the Bockstein homomorphism $\beta\colon H_2(M;\Q/\Z)\to H_1(M;\Z)$  as a connecting homomorphism and the above properties of $x$ and $z$, it follows immediately that
\[  \beta([\rho_*(z+x)])\,=\, i_*(\Phi_{\Z}(v))\,\in\, H_1(M;\Q/\Z).\] 

To conclude the proof of the lemma, recall that the map $\theta$ is defined as the composition 
\[ H_1(M;\Z)\, \xrightarrow{\beta^{-1}}\, H_2(M;\Q/\Z)\, \xrightarrow{\,p_*\,} \,H_2(M,X;\Q / \Z) \,\xleftarrow{\cong}\, H_2(Z,F; \Q / \Z)\, \xrightarrow{\,\partial_{\Q/\Z}\,} \,H_1(F; \Q / \Z).\]
Using the definition of the relative homology group $H_2(M,X; \Q / \Z)$ and the previous computation, it follows that 
\[ (p_* \circ \beta^{-1} \circ i_*(\Phi_{\Z}(v)))\,\,=\,\,p_*([\rho_*(z+x)])\,\,=\,\,\rho_*([z]).\] Since $z$ is already a singular chain in $C_2(Z,F;\Q/\Z)$ it suffices to prove the following claim.

\begin{claim}
We have $\partial_{\Q/\Z}(\rho_*([z]))=-\Psi_{\Q/\Z}(B^{-1}Av)\in H_1(F;\Q/\Z)$.
\end{claim}

By the choice of $z$ we have 
$\partial_\Q(z)=\wti{\Phi}_\Z(v)-\wti{\Psi}_{\Q}(B^{-1}Av) \in C_1(F;\Q)$. 
This implies that 
$\partial_{\Q/\Z}(\rho_*([z]))=[\rho_*(\wti{\Phi}_\Z(v))-\rho_*(\wti{\Psi}_{\Q}(B^{-1}Av))] \in H_1(F;\Q/\Z)$.
But $\wti{\Phi}_{\Z}(v)$ is an integral class, so we have 
$\partial_{\Q/\Z}(\rho_*([z]))=-\Psi_{\Q/\Z}(B^{-1}Av)\in H_1(F;\Q/\Z)$. This concludes the proof of the proposition.
\end{proof}

We can now provide  the proof of Theorem~\ref{thm:Main}. In fact we will prove the following slightly more precise statement:

\begin{theorem}
\label{thm:Main2}
Let $g\in \N$ and let $\varphi\colon F_g\to F_g$ be an orientation-preserving diffeomorphism.  Suppose that $M(\varphi)$ is a rational homology sphere.
Then the following statements hold:
\bnm[font=\normalfont]
\item The above
homomorphism $i_*\circ \Phi_{\Z}\colon \Z^g\to H_1(M(\varphi);\Z)$  descends to an isomorphism
\[ \hspace{1cm} i_*\circ \Phi\colon \Z^g/B_\varphi^T \Z^g\,\,\xrightarrow{\,\cong\,}\,\,H_1(M(\varphi);\Z),\]
in particular the matrix $B_\varphi\in M(g\times g,\Z)$ has non-zero determinant.
\item The isomorphism $\Phi$ from $(1)$ defines an isometry from  the  form
\begin{align*}
\hspace{1cm}
\Z^g / B_\varphi^T \Z^g \times \Z^g / B_\varphi^T \Z^g &\to\,\, \Q / \Z \\
(v,w) &\mapsto \,\,-v^T  B_\varphi^{-1}A_\varphi w
\end{align*}
to the linking form of $M(\varphi)$.
\enm
\end{theorem}

\begin{remark}
As a reality check it is worth verifying that the form given in Theorem~\ref{thm:Main2} (2) is actually well-defined. It is clear that the form does not depend on the choice of the  representative  $w$. Furthermore, 
by Lemma~\ref{lem:MatrixProperties} we have $AB^T=BA^T$ which implies that the form does not depend on the choice of the representative $v$.
\end{remark}

\begin{proof}
Note that statement (1) has already been proved in  Proposition~\ref{prop:Presentation}.
Therefore it is enough to show that for all $v,w\in \Z^g$ we have
\[ \lambda_M(i_*(\Phi_\Z(v)),i_*(\Phi_\Z(w)))\,\,=\,\,v^T(B^{-1}A)w\,\in\, \Q/\Z.\]
Combining~(\ref{eq:Reduction}) with Proposition~\ref{prop:ComputeTheta}, we obtain the equality
\begin{equation}\label{eq:proof1}
\ba{rl}
\lambda_M(i_*(\Phi_\Z(v)),i_*(\Phi_\Z(w)))
&=\,\,\big\langle (\op{PD}_F^{\Q / \Z} \circ \theta \circ i_* )(\Phi_\Z(v)), \Phi_\Z(w) \big\rangle_F \\
&=\,\,-\big\langle \op{PD}_F^{\Q / \Z} (\Psi_{\Q/\Z}(B^{-1}Av)),\Phi_\Z(w)\big\rangle_F.\ea
\end{equation}
The commutativity of the diagram 
$$ \xymatrix@C1.3cm@R0.5cm{
H_1(F;\Q) \ar[r]^-{\op{PD}_F^\Q} \ar[d]^(0.45){\rho_*}& H^1(F;\Q) \ar[r]^-{\op{ev}} \ar[d]^(0.45){\rho_*}& \op{Hom}(H_1(F;\Q),\Q)\ar[d]^(0.45){\rho_*} \\
H_1(F;\Q / \Z) \ar[r]^-{\op{PD}_F^{\Q / \Z}} & H^1(F;\Q / \Z)  \ar[r]^-{\op{ev}} & \op{Hom}(H_1(F;\Q),\Q / \Z) }$$
now implies that
\begin{equation}\label{eq:proof2}
\big\langle \op{PD}_F^{\Q / \Z} (\Psi_{\Q/\Z}(B^{-1}Av)),\Phi_\Z(w)\big\rangle_F\,\,=\,\, \rho_* \big(\big\langle  \op{PD}_F^\Q (\Psi_\Q(B^{-1}Av)),\Phi_\Q(w) \big\rangle_F\big).
\end{equation}
By the calculation of the intersection form of the surface $F$ given in (\ref{eq:intersection-form-surface}) we have
\begin{equation}\label{eq:proof3}
\rho_* \big(\big\langle  \op{PD}_F^\Q (\Psi_\Q(B^{-1}Av)),\Phi_\Q(w) \big\rangle_F\big)\,\,=\,\,-Q_F\big(\Psi_\Q(B^{-1}Av),\Phi_\Q(w)\big).\end{equation}
Finally we recall that the $a_r$ and $b_r$ form a symplectic basis for $H_1(F;\Z)$, i.e.\ with respect to this basis the intersection form $Q_F$ is represented by the matrix $\left(\begin{smallmatrix}0&I_g\\ -I_g&0\end{smallmatrix}\right)$.  
In our context, together with the equality $AB^T=BA^T$ from Lemma~\ref{lem:MatrixProperties} this implies that
\begin{equation}\label{eq:proof4}
 Q_F\big(\Psi_\Q(B^{-1}Av),\Phi_\Q(w)\big)=-
\bpp 0 \\  B^{-1}Av\epp^T\hspace{-0.1cm} \bpp 0&I_g\\ -I_g&0\epp \hspace{-0.1cm}\bpp w \\ 0\epp=
 v^TA^T(B^{-1})^Tw= v^TB^{-1}Aw.\end{equation}
The desired statement now follows from the combination of 
(\ref{eq:proof1}), (\ref{eq:proof2}), (\ref{eq:proof3}) and (\ref{eq:proof4}).
\end{proof}

\bibliography{BiblioLinking}
\bibliographystyle{plainnat}

\end{document}